\documentclass[10pt]{amsart}
\usepackage{amssymb}

\usepackage{blindtext,stmaryrd}
\usepackage[utf8]{inputenc}
\newtheorem{thm}{Theorem}[section]

\newtheorem{cor}[thm]{Corollary}
\newtheorem{lem}[thm]{Lemma}
\newtheorem{prop}[thm]{Proposition}
\theoremstyle{definition}
\newtheorem{defn}[thm]{Definition}

\theoremstyle{remark}

\numberwithin{equation}{section}

\begin{document}

\title{On extensions of partial isomorphisms} 

\author{Mahmood Etedadialiabadi}

\address{Department of Mathematics, University of North Texas, 1155 Union Circle \#311430, Denton, TX 76203, USA}

\email{mahmood.etedadialiabadi@unt.edu}

\author{Su Gao}

\address{Department of Mathematics, University of North Texas, 1155 Union Circle \#311430, Denton, TX 76203, USA}\email{sgao@unt.edu}

\thanks{The second author's research was partially supported by the NSF grant DMS-1800323.}

\subjclass[2010]{Primary  03C13,03C55; Secondary 20E06,20E26}

\keywords{Hrushovski property, extension property for partial automorphisms (EPPA), partial isomorphism, HL-extension, HL-map, coherent, ultraextensive, ultrahomogeneous, locally finite, Henson Graph}


\begin{abstract}
In this paper we study a notion of HL-extension (HL standing for Herwig--Lascar) for a structure in a finite relational language $\mathcal{L}$. We give a description of all finite minimal HL-extensions of a given finite $\mathcal{L}$-structure. In addition, we study a group-theoretic property considered by Herwig--Lascar and show that it is closed under taking free products. We also introduce notions of coherent extensions and ultraextensive $\mathcal{L}$-structures and show that every countable $\mathcal{L}$-structure can be extended to a countable ultraextensive structure. Finally, it follows from our results that the automorphism group of any countable ultraextensive $\mathcal{L}$-structure has a dense locally finite subgroup.
\end{abstract}

\maketitle
\bigskip\noindent
\section{Introduction}

Let $C_1,C_2$ be two structures in a given relational language $\mathcal{L}$. A \textbf{partial isomorphism} from $C_1$ into $C_2$ is an isomorphism of a substructure of $C_1$ onto a substructure of $C_2$. A \textbf{partial automorphism} (or a partial isomorphism) of an $\mathcal{L}$-structure $C$ is an isomorphism between two (possibly different) substructures of $C$.

\begin{defn}
Let $\mathcal{C}$ be a class of $\mathcal{L}$-structures (containing both finite and infinite structures). $\mathcal{C}$ is said to have \textbf{the extension property for partial automorphisms (EPPA)} if whenever $C_1$ and $C_2$ are structures in $\mathcal{C}$, $C_1$ is finite, $C_1$ is a substructure of $C_2$, and every partial automorphism of $C_1$ extends to an automorphism of $C_2$, then there exist a finite structure $C_3$ in $\mathcal{C}$ which extends $C_1$ and every partial automorphism of $C_1$ extends to an automorphism of $C_3$. 
\end{defn}
Hrushovski [\ref{Hru_B}], was one of the first papers to consider the question of whether a certain class of structures has the EPPA. More precisely, he showed that the class of simple graphs has the EPPA, that is, every finite graph $G$ can be extended to another finite graph, $H$, such that every partial isomorphism of $G$ extends to an automorphism of $H$. Herwig--Lascar [\ref{HL_B}], generalized the result of Hrushovski to finite relational structures. 
\begin{defn}
If $M$ is an $\mathcal{L}$-structure and $\mathcal{T}$ a set of $\mathcal{L}$-structures, we say that $M$ is \textbf{$\mathcal{T}$-free} if there is no structure $T\in\mathcal{T}$ and homomorphism $h :T\rightarrow M$.
\end{defn}

Here we use the same definition of a homomorphism as in [\ref{HL_B}]. That is, if $M$ and $N$ are $\mathcal{L}$-structures, a homomorphism from $M$ to $N$ is a map $h: M\to N$ such that, if $n$ is an integer, $R$ is an $n$-ary relation symbol of $\mathcal{L}$, and $a_1, a_2,\dots,a_n$ are elements of $M$ with $M\vDash R(a_1,a_2,\dots, a_n)$, then $N\vDash R(h(a_1),h(a_2),\dots, h(a_n))$.

\begin{thm}[Herwig--Lascar \cite{HL}]\label{thm_HL}
Let $\mathcal{L}$ be a finite relational language and $\mathcal{T}$ a finite set of finite $\mathcal{L}$-structures. Then the class of all finite $\mathcal{T}$-free $\mathcal{L}$-structures has the EPPA.
\end{thm}

Inspired by the Herwig--Lascar theorem, we define the following notions. Let $\mathcal{L}$ be a finite relational language and $C$ be an $\mathcal{L}$-structure. An {\it HL-extension} of $C$ is a pair $(D,\phi)$, where $D$ is an $\mathcal{L}$-structure extending the structure $C$, and $\phi$ is a map from the set of all partial isomorphisms of $C$ into the set of all automorphisms of $D$ such that $\phi(p)$ extends $p$. With this notion, the Herwig--Lascar theorem can be restated as: Every finite $\mathcal{T}$-free $\mathcal{L}$-structure has a finite $\mathcal{T}$-free HL-extension.

If $C$ is an $\mathcal{L}$-structure and $(D, \phi)$ is an HL-extension of $C$, then we say $(D,\phi)$ is {\it minimal} if for all $b\in D$, there are partial isomorphisms $p_1,\dots, p_n$ of $C$ and $a\in C$ such that
$$ b=\phi(p_1)\cdots\phi(p_n)(a). $$

Our first main result of the paper is a description of all finite $\mathcal{T}$-free, minimal HL-extensions of a given finite $\mathcal{T}$-free $\mathcal{L}$-structure. To do this, we describe a canonical collection of finite $\mathcal{T}$-free, minimal HL-extensions from the original construction of Herwig--Lascar \cite{HL}, and show that every other finite $\mathcal{T}$-free, minimal HL-extension is a homomorphic image of one of the canonical extensions.  

Our next result is regarding a group-theoretic property in the profinite topology considered by Herwig--Lascar in \cite{HL}. We call it the \textit{HL-property}. For comparison, we say that a group $G$ has the \textit{RZ-property} (RZ standing for Ribes--Zalesskii) if any finite product of finitely generated subgroups of $G$ is closed in the profinite topology. Every group with the RZ-property is residually finite. Ribes-Zalesskii \cite{RZ} proved the RZ-property for finitely generated free groups. Herwig--Lascar \cite{HL} introduced the HL-property as a strengthening of the RZ-property, and showed that the Herwig--Lascar theorem is essentially equivalent to the HL-property for finitely generated free groups. Coulbois \cite{C} gave a characterization of the RZ-property in terms of extensions of partial isomorphisms and used it to show that the RZ-property is preserved under taking free products. Rosendal \cite{R} gave a characterization of the RZ-property in terms of extensions of partial isometries for finite metric spaces. Here we give a similar characterization for the HL-property of groups, and show that the HL-property is also preserved under taking free products.

In \cite{S}, Solecki proved the EPPA for the class of finite metric spaces. Furthermore, Siniora--Solecki \cite{S2}  proved a stronger version of the Herwig--Lascar theorem. They showed that for a structure $C$ with an HL-extension one can find an HL-extension $(D,\phi)$ with the property that for every triple $(p,q,r)$ of partial isomorphisms of $C$ with $p=q\circ r$ we have $\phi(p)=\phi(q)\circ\phi(r)$. This property has been referred to as \textit{coherence}. A similar concept was considered in \cite{HKN} and \cite{R}.

In this paper we introduce a slightly different notion of coherence between HL-extensions. If $C_1\subseteq C_2$ are $\mathcal{L}$-structures, $(D_1,\phi_1)$ is an HL-extension of $C_1$, and $(D_2, \phi_2)$ is an HL-extension of $C_2$, then we say that $(D_1,\phi_1)$ and $(D_2, \phi_2)$ are {\it coherent} if $D_2$ extends $D_1$, $\phi_2(p)$ extends $\phi_1(p)$ for every partial isomorphism $p$ of $C_1$, and the map $\phi_1(p)\mapsto\phi_2(p)$, where $p$ ranges over all partial isomorphisms of $C_1$, induces an isomorphism between a subgroup of all automorphisms of $D_1$ and a subgroup of all automorphisms of $D_2$. We will show that, if $\mathcal{T}$ is a finite set of $\mathcal{L}$-structure each of which is a Gaifman clique, then for any $C_1\subseteq C_2$ finite $\mathcal{T}$-free $\mathcal{L}$-structures and $(D_1, \phi_1)$ a finite $\mathcal{T}$-free HL-extension of $C_1$, there is a finite $\mathcal{T}$-free HL-extension of $C_2$ coherent with $(D_1,\phi_1)$. In the proof of this result we use the above-mentioned coherence result of Siniora--Solecki \cite{S2}. We should also mention that Hubi\v{c}ka, Kone\v{c}n\`{y}, and Ne\v{s}et\v{r}il  \cite{HKN} presented a direct combinatorial construction of HL-extensions with the same coherence property. The technical assumption in the theorem about $\mathcal{T}$ is necessary and optimal for the proof. 

We call an $\mathcal{L}$-structure $U$ {\it ultraextensive} if $U$ is ultrahomogeneous, every finite $C\subseteq U$ has a finite HL-extension $(D, \phi)$ where $D\subseteq U$, and if $C_1\subseteq C_2\subseteq U$ are finite and $(D_1,\phi_1)$ is a finite minimal HL-extension of $C_1$ with $D_1\subseteq U$, then there is a finite minimal HL-extension $(D_2, \phi_2)$ of $C_2$ such that $D_2\subseteq U$ and $(D_1,\phi_1)$ and $(D_2,\phi_2)$ are coherent.

Recall that ultrahomogeneity means that any finite partial isomorphism can be extended to an isomorphism of the entire space. Thus ultraextensiveness is a strengthening of ultrahomogeneity. We will establish the following results about ultraextensive $\mathcal{L}$-structures.

\begin{thm} 
Every countable $\mathcal{L}$-structure can be extended to a countable ultraextensive $\mathcal{L}$-structure. Moreover, if $\mathcal{T}$ is a finite set of finite $\mathcal{L}$-structures each of which is a Gaifman clique, then every countable $\mathcal{T}$-free $\mathcal{L}$-structure can be extended to a countable $\mathcal{T}$-free ultraextensive $\mathcal{L}$-structure.
\end{thm}

\begin{thm} 
If $U$ is an ultraextensive $\mathcal{L}$-structure then every countable substructure $C\subseteq U$ can be extended to a countable ultraextensive substructure $D\subseteq U$.
\end{thm}

\begin{thm} 
If $U$ is a countable ultraextensive $\mathcal{L}$-structure then the automorphism group of $U$ has a dense locally finite subgroup.
\end{thm}

The rest of the paper is organized as follows. In Section~\ref{Preliminaries} we give the characterization of finite $\mathcal{T}$-free, minimal HL-extensions. In Section~\ref{HL-property} we study the HL-property of groups and show that it is preserved under taking free products. In Section~\ref{CoherentHLextensions} we discuss coherent HL-extensions and ultraextensive structures. The results in Sections ~\ref{Preliminaries} and \ref{CoherentHLextensions} are analogous to previous work by the authors \cite{EG} on similar concepts in the context of metric spaces. 


\section{Minimal HL-Extensions\label{Preliminaries}}

\subsection{HL-extensions}

We fix some notation to be used in the rest of the paper. Throughout this paper let $\mathcal{L}$ be a finite relational language. Let $C,D$ be $\mathcal{L}$-structures. We say that $D$ is an {\it extension} of $C$ if $C$ is a substructure of $D$. Interchangeably, we use the same terminology when $D$ contains an isomorphic copy of $C$.

A {\it homomorphism} from $C$ to $D$ is a map $\pi: C\to D$ such that for every $n$-ary relation $R\in\mathcal{L}$ and every $a_1,\dots, a_n\in C$,
$$ R^C(a_1,\dots, a_n)\Rightarrow R^D(\pi(a_1),\dots,\pi(a_n)). $$
An {\it isomorphism} from $C$ to $D$ is a bijection $\pi: C\to D$ such that for every $n$-ary relation $R\in \mathcal{L}$ and every $a_1,\dots,a_n\in C$, 
$$ R^C(a_1,\dots,a_n) \iff R^D(\pi(a_1),\dots,\pi(a_n)). $$
An isomorphism from $C$ to $C$ is also called an {\it automorphism of} $C$. The set of all automorphisms of $C$ is denoted as $\mbox{Aut}(C)$. Under composition of maps, $\mbox{Aut}(C)$ becomes a group.

A {\it partial isomorphism} of $C$ is an isomorphism between two finite substructures of $C$. The set of all partial isomorphisms of $C$ is denoted as $\mathcal{P}(C)$. Although $\mathcal{P}(C)$ is not necessarily a group, it is a groupoid and for each $p\in \mathcal{P}(C)$ we can speak of the inverse map $p^{-1}$, which is still a partial isomorphism. 

If $D$ is an extension of $C$, then every partial isomorphism of $C$ is also a partial isomorphism of $D$. In symbols, we have $\mathcal{P}(C)\subseteq \mathcal{P}(D)$ if $C$ is a substructure of $D$.

If $p, q\in \mathcal{P}(C)$, we say that $q$ {\it extends} $p$, and write $p\subseteq q$, if 
$$\{ (a, p(a))\,:\, a\in\mbox{dom}(p)\}\subseteq \{ (a, q(a))\,:\, a\in\mbox{dom}(q)\}.$$

We let $1_C$ denote the identity automorphism on $C$, i.e., $1_C(a)=a$ for all $a\in C$. Let $\mathcal{P}_C$ denote the set of all $p\in \mathcal{P}(C)$ such that $p\not\subseteq 1_C$. We refer to elements of $\mathcal{P}_C$ as {\it nonidentity partial isomorphisms} of $C$. Note that if $p\in \mathcal{P}_C$ then $p^{-1}\in \mathcal{P}_C$ and $p^{-1}\neq p$.

The main concept we study in this paper is that of an HL-extension.

\begin{defn} Let $C$ be an $\mathcal{L}$-structure. An {\it HL-extension} of $C$ is a pair $(D, \phi)$, where $D$ is an extension of $C$, and $\phi: \mathcal{P}_C\to \mbox{Aut}(D)$ such that $\phi(p)$ extends $p$ for all $p\in\mathcal{P}_C$. 
\end{defn}

Note that if $(D, \phi)$ is an HL-extension of $C$ then we can always modify $\phi$ so that for all $p\in \mathcal{P}_C$, $\phi(p^{-1})=\phi(p)^{-1}$. We will tacitly assume this property for all the HL-extensions we consider.

Note that an equivalent restatement of Herwig--Lascar theorem (Theorem~\ref{thm_HL}) is that every finite $\mathcal{T}$-free $\mathcal{L}$-structure has a finite $\mathcal{T}$-free HL-extension. 

We will need the following notion of homomorphism between HL-extensions.

\begin{defn} Let $C$ be an $\mathcal{L}$-structure, and let $(D_1, \phi_1)$ and $(D_2, \phi_2)$ be both HL-extensions of $C$. A {\it homomorphism} from $(D_1, \phi_1)$ to $(D_2, \phi_2)$ is a map $\psi: D_1\to D_2$ such that $\psi$ is a homomorphism from the structure $D_1$ to $D_2$, $\psi\upharpoonright C$ is the identity map on $C$, and for all $p\in \mathcal{P}_C$, $\psi\circ\phi_1(p)=\phi_2(p)\circ \psi$. 
\end{defn}

We also define the notion of minimality for an HL-extension as follows.

\begin{defn} Let $C$ be an $\mathcal{L}$-structure and $(D,\phi)$ be an HL-extension of $C$. We say that $(D,\phi)$ is \textit{minimal} if for all $b\in D\setminus C$ there are $p_1, \dots, p_n\in \mathcal{P}_C$ and $a\in C$ such that $b=\phi(p_1)\dots \phi(p_n)(a)$. 
\end{defn}

\subsection{A canonical HL-extension} 
In this subsection we describe a canonical construction of an HL-extension that is essentially due to Herwig--Lascar \cite{HL}. In the rest of the paper let $\mathcal{T}$ be a fixed finite set of finite $\mathcal{L}$-structures.

First, note that for every finite $\mathcal{L}$-structure $C$ there is a unique partition of $C$ into substructures $\{C_i \ : \ i=1,\dots,n\}$ such that each $C_i$ is a maximal subset of $C$ satisfying that for every $a,b\in C_i$, the map that sends $a$ to $b$ (that is, the map $\{(a,b)\}$) is a partial isomorphism of $C$. In other words, we partition $C$ into maximal subsets whose elements satisfy the same unary predicates. We call this partition with a specific point from each set a \textit{natural factorization} of $C$.  That is, a natural factorization of $C$ is of the form $\{(C_i,a_i): i=1,\dots, n\}$, where each $a_i\in C_i$.

Let $C$ be a finite $\mathcal{T}$-free $\mathcal{L}$-structure. Let $\{(C_i,a_i) \ : \ i=1,\dots,n\}$ be a natural factorization of $C$. For every $1\leq i\leq n$ we define
\[
H_i=\{g\in \mathbb{F}(\mathcal{P}_C) \ : \ g(a_i)=a_i\},
\]
where $\mathbb{F}(\mathcal{P}_C)$ is the free group with the generating set $\mathcal{P}_C$ (with the convention that the inverse of $p\in\mathcal{P}_C$ in $\mathbb{F}(\mathcal{P}_C)$ coincides with $p^{-1}$). By $g(a_i)=a_i$ we mean that if $g=p_1\cdots p_m$ with $p_1,\dots, p_m\in\mathcal{P}_C$, then $p_1(\cdots( p_m(a_i))\cdots)$ is defined and $$p_1(\cdots (p_m(a_i))\cdots)=a_i.$$ 
Each $H_i$ is a subgroup of $\mathbb{F}(\mathcal{P}_C)$. 

Let $\Gamma$ be the $\mathcal{L}$-structure with domain $$ \mathbb{F}(\mathcal{P}_C)/H_1\sqcup \cdots \sqcup \mathbb{F}(\mathcal{P}_C)/H_n$$ and such that for every $m$-ary relation symbol $R\in\mathcal{L}$, we have $R^\Gamma(g_1H_{i_1},\dots,g_mH_{i_m})$ iff there are $p_1,\dots,p_m\in \mathcal{P}_C$ and $g\in \mathbb{F}(\mathcal{P}_C)$ such that $p_j(a_{i_j})$ is defined for each $j=1,\dots, m$, $(g_1H_{i_1},\dots,g_mH_{i_m})=(gp_1H_{i_1},\dots,gp_mH_{i_m})$, and $R^C(p_1(a_{i_1}),\dots,p_m(a_{i_m}))$.

Note that $C$ can be viewed as a substructure of $\Gamma$. In fact, consider the map $\pi:C\rightarrow \Gamma$ defined as
$$ \pi(a)=\left\{\begin{array}{ll} H_i, & \mbox{ if $a=a_i$,} \\ pH_i, & \mbox{ if $a\neq a_i$ and $p\in \mathcal{P}_C$ satisfies $p(a_i)=a$,}  \end{array}\right. $$
for $a\in C_i$. It is easy to see that $\pi$ is well-defined and is indeed an isomorphic embedding from $C$ into $\Gamma$. 

Given any $\gamma\in \mathbb{F}(\mathcal{P}_C)$, the map $\Phi_\gamma$ defined by $\Phi_\gamma(gH_i)=\gamma gH_i$ is an automorphism of $\Gamma$. Thus, $(\Gamma, \Phi)$ is an HL-extension of $C$ with $\Phi: \mathcal{P}_C\to \mbox{Aut}(\Gamma)$ defined as $\Phi(p)=\Phi_p$. Note that by definition, $(\Gamma,\Phi)$ is a minimal HL-extension of $C$.

Assume $C$ has a $\mathcal{T}$-free HL-extension $(D,\phi)$. Consider the map $\psi:\Gamma \rightarrow D$ defined by $\psi(gH_i)=\phi(g)(a_i)$, where $\phi(g)=\phi(p_1)\cdots\phi(p_m)$ if $g=p_1\dots p_m$. Then $\psi$ is a homomorphism. It follows that $\Gamma$ is also $\mathcal{T}$-free.

\subsection{Finite HL-extensions} Let $C$ be a finite $\mathcal{T}$-free $\mathcal{L}$-structure as before. We give a description of all finite $\mathcal{T}$-free, minimal HL-extensions of $C$. For this we first describe a finite $\mathcal{T}$-free, minimal HL-extension by replacing each group $H_i$ in the above canonical $\Gamma$ with a larger group of the form $N_iH_i$, where $N_i$ is a normal subgroup of $\mathbb{F}(\mathcal{P}_C)$ of finite index. 

Let $N_1,\dots, N_n\unlhd \mathbb{F}(\mathcal{P}_C)$ be normal subgroups of finite index. We define
$$ \Gamma_{\vec{N}}=\mathbb{F}(\mathcal{P}_C)/N_1H_1\sqcup \cdots \sqcup \mathbb{F}(\mathcal{P}_C)/N_nH_n. $$
The structure on $\Gamma_{\vec{N}}$ is defined analogously to the structure on the canonical HL-extension $\Gamma$. More precisely, to define the structure on $\Gamma_{\vec{N}}$, let $R\in \mathcal{L}$ be an $m$-ary relation symbol. Then $R^{\Gamma_{\vec{N}}}(g_1N_{i_1}H_{i_1},\dots,g_mN_{i_m}H_{i_m})$ iff there are $p_1,\dots,p_m\in \mathcal{P}_C$ and $g\in \mathbb{F}(\mathcal{P}_C)$ such that $p_j(a_{i_j})$ is defined for each $j=1,\dots, m$, 
\[
(g_1N_{i_1}H_{i_1},\dots,g_mN_{i_m}H_{i_m})=(gp_1N_{i_1}H_{i_1},\dots,gp_mN_{i_m}H_{i_m}),
\]
and $R^C(p_1(a_{i_1}),\dots,p_m(a_{i_m}))$.

Consider the map $\pi_{\vec{N}}:C\rightarrow \Gamma_{\vec{N}}$ defined by 
$$ \pi_{\vec{N}}(a)=\left\{\begin{array}{ll} N_iH_i, & \mbox{ if $a=a_i$,} \\ pN_iH_i, & \mbox{ if $a\neq a_i$ and $p\in \mathcal{P}_C$ satisfies $p(a_i)=a$,}  \end{array}\right. $$
for $a\in C_i$. Then $\pi_{\vec{N}}$ is well-defined. Under suitable assumptions (that will be discussed in Theorem \ref{minimalHLextension}), $\pi_{\vec{N}}$ becomes an isomorphic embedding. In this case $\pi_{\vec{N}}(C)$ is an isomorphic copy of $C$ as a substructure of $\Gamma_{\vec{N}}$. 

We define $\Phi_{\vec{N}}:\mathcal{P}_C\rightarrow \mbox{Aut}(\Gamma_{\vec{N}})$ by letting 
$$\Phi_{\vec{N}}(p)(gN_iH_i)=p gN_iH_i. $$
Assuming the above map $\pi_{\vec{N}}$ is an isomorphic embedding, and noting that there is a canonical surjective homomorphism from $\Gamma$ to $\Gamma_{\vec{N}}$, it follows from the minimality of $(\Gamma, \Phi)$ that $(\Gamma_{\vec{N}}, \Phi_{\vec{N}})$ is also a minimal HL-extension of $C$.  

We are now ready to describe any finite $\mathcal{T}$-free, minimal HL-extension of $C$ as a homomorphic image of some $(\Gamma_{\vec{N}}, \Phi_{\vec{N}})$, which is itself a finite $\mathcal{T}$-free, minimal HL-extension of $C$.

\begin{thm}\label{minimalHLextension}
Let $C$ be a finite $\mathcal{T}$-free $\mathcal{L}$-structure and $(D,\phi)$ be a finite $\mathcal{T}$-free, minimal HL-extension of $C$. Then, there are $N_1,\dots,N_n\unlhd \mathbb{F}(\mathcal{P}_C)$ of finite index such that $(\Gamma_{\vec{N}}, \Phi_{\vec{N}})$ is a finite $\mathcal{T}$-free, minimal HL-extension of $C$ and there is a homomorphism from $(\Gamma_{\vec{N}}, \Phi_{\vec{N}})$ onto $(D,\phi)$.
\end{thm}

\begin{proof}
For each $i=1,\dots, n$, let $D_i=\{\phi(g)(a_i): g\in \mathbb{F}(\mathcal{P}_C)\}$. We define
\[
N_i=\{g\in \mathbb{F}(\mathcal{P}_C)  : \phi(g)(a)=a \text{ for every } a\in D_i \}.
\]
Then $N_i\unlhd \mathbb{F}(\mathcal{P}_C)$. Since $D$ is finite, each $N_i$ is of finite index. 

Let $\Gamma_{\vec{N}}$ and $\pi_{\vec{N}}: C\to \Gamma_{\vec{N}}$ be defined as above. We claim that $\pi_{\vec{N}}$ is an isomorphic embedding. To see this, let $a, a'\in C_i$ with $a\neq a'$. Let $p, p'\in\mathcal{P}_C$ with $p(a_i)=a$ and $p'(a_i)=a'$. We show that $p'^{-1}pH_i\cap N_i=\emptyset$, which implies $pN_iH_i\neq p'N_iH_i$. For this, let $g\in H_i$. Since
$$\phi(p'^{-1}pg)(a_i)=\phi(p'^{-1})\phi(p)\phi(g)(a_i)=\phi(p')^{-1}p(a_i)\neq a_i, $$
we have that $p'^{-1}pH_i\cap N_i=\emptyset$. This shows that $\pi_{\vec{N}}$ is injective. It is easy to see that $\pi_{\vec{N}}$ is an isomorphism between the structures $C$ and $\pi_{\vec{N}}(C)$.

Now we define $\psi_{\vec{N}}:\Gamma_{\vec{N}}\rightarrow D$ by $\psi_{\vec{N}}(gN_iH_i)=\phi(g)(a_i)$. Note that $\psi_{\vec{N}}$ is well-defined since if $g_1^{-1}g_2\in N_iH_i$ then by definition of $N_i,H_i$ we have $\phi(g_1^{-1}g_2)(a_i)=a_i$ and therefore $\phi(g_1)(a_i)=\phi(g_2)(a_i)$. $\psi_{\vec{N}}$ is onto since $D$ is minimal. It is also easy to verify that $\psi_{\vec{N}}$ is a homomorphism. It follows that $\Gamma_{\vec{N}}$ is $\mathcal{T}$-free, and thus $(\Gamma_{\vec{N}}, \Phi_{\vec{N}})$ is a finite $\mathcal{T}$-free, minimal HL-extension of $C$.

Finally, it is routine to check that for every $p\in \mathcal{P}_C$, $\phi(p)\circ\psi_{\vec{N}}=\psi_{\vec{N}}\circ \Phi_{\vec{N}}(p)$. Thus $\psi_{\vec{N}}$ is a homomorphism from $(\Gamma_{\vec{N}}, \Phi_{\vec{N}})$ onto $(D, \phi)$.
\end{proof}

\section{The HL-Property of a Group\label{HL-property}}
In this section we consider a property for a group $G$ analogous to the existence of HL-extensions for free groups. 

\subsection{The HL-property}
First, we need the following definitions.
\begin{defn}[Herwig--Lascar \cite{HL}]
Let $G$ be a group and let $H_1,\dots,H_n\leq G$. A \textit{left system} of equations on $H_1,\dots,H_n$ is a finite set of equations with variables $x_1,\dots,x_m$ and constants $g_1,\dots,g_l$ such that each equation is of the form 
\[
x_iH_j=g_kH_j \text{ or } x_iH_j=x_rg_kH_j
\]
where $1\leq i,r\leq m$, $1\leq k\leq l$ and $1\leq j\leq n$. 
\end{defn} 

\begin{defn}
Let $G$ be a group. We say that $G$ has the \textit{HL-property} if for every finitely generated $H_1,\dots,H_n\leq G$ and left system of equations on $H_1,\dots,H_n$ that does not have a solution, there exist normal subgroups of finite index $N_1,\dots,N_n\unlhd G$ such that the same left system of equations on $N_1H_1,\dots,N_nH_n$ does not have a solution.
\end{defn}

By results of \cite{HL}, Section $3$, the Herwig--Lascar theorem (Theorem~\ref{thm_HL}) implies the HL-property for all free groups with finitely many generators. Our results below will imply that they are actually equivalent.

Recall that we say a group $G$ has the {\it RZ-property} if for any finitely generated subgroups $H_1, \dots, H_n\leq G$, $H_1\cdots H_n$ is closed in the profinite topology of $G$. Equivalently, a group $G$ has the RZ-property iff for any finitely generated $H_1,\dots,H_n\leq G$ and $g\notin H_1\cdots H_n$ there exist a normal subgroup $N\unlhd G$ of finite index, such that $gN\cap H_1\cdots H_n=\emptyset$. Ribes--Zalesskii \cite{RZ} proved the RZ-property for free groups with finitely many generators. As noted in \cite{HL}, the HL-property is a strengthening of the RZ-property, and therefore the Herwig--Lascar theorem is a strengthening of the Ribes--Zalesskii result. 

Rosendal in \cite{R} considered the RZ-property and showed that it is equivalent to a statement about extensions of partial isometries for finite metric spaces which he called \textit{finite approximability}. Earlier, Coulbois \cite{C} gave a characterization of the RZ-property in terms of extensions of partial isomorphisms of finite structures, and used it to show that the RZ-property is closed under taking finite free products. Below we give a characterization of the HL-property also in terms of extensions of partial isomorphisms of finite structures. Our characterization is analogous to Rosendal's notion of finite approximability. 

To state the theorem, we need the following notions.  Let $G$ be a group acting on sets $X$ and $Y$ and let $A\subseteq X$ and $F\subseteq G$ be arbitrary subsets. An \textit{$F$-map} from $A$ to $Y$ is a function $\pi: A\to Y$ such that for all $g\in F$ and $x\in A$, if $g(x)\in A$, then $\pi(g(x))=g(\pi(x))$. Moreoever, if $X$ and $Y$ are $\mathcal{L}$-structures, then $\pi$ is called an \textit{$F$-embedding} if $\pi$ is an injective $F$-map that is an isomorphism between $A$ and $\pi(A)$. 

An $\mathcal{L}$-structure $C$ is called a \textit{Gaifman clique} if for every $a, b\in C$ there is a relation symbol $R\in \mathcal{L}$ with arity $m\geq 2$ and $c_1,\dots, c_m\in C$ with $a, b\in \{c_1, \dots, c_m\}$ and $R^C(c_1,\dots, c_m)$. 

\begin{thm}\label{HLproperty}
Let $G$ be a group. Then the following are equivalent:
\begin{enumerate}
\item[\rm (i)] $G$ has the HL-property;
\item[\rm (ii)] Let $\mathcal{L}$ be a finite relational language with unary relation symbols $S_1,\dots, S_n\in\mathcal{L}$. Let $\mathcal{T}$ be a finite set of finite $\mathcal{L}$-structures. Let $D$ be a $\mathcal{T}$-free $\mathcal{L}$-structure such that $\{S_1^D, \dots, S_n^D\}$ is a partition of the domain of $D$. Let $C$ be a finite substructure of $D$. Let $F$ be a finite subset of $G$. Suppose that $G$ acts faithfully by isomorphisms on $D$ and that $G$ acts transitively on each $S_i^D$ for $i=1,\dots, n$. Then there exists a finite $\mathcal{T}$-free $\mathcal{L}$-structure $D'$ on which $G$ acts by isomorphisms, and an $F$-embedding from $C$ to $D'$. 
\item[\rm (iii)] Clause {\rm (ii)} with the additional assumption that every structure $T\in\mathcal{T}$ is a Gaifman clique.
\end{enumerate}
\end{thm}

The next two subsections are devoted to a proof of Theorem~\ref{HLproperty}. We will show (i)$\Rightarrow$(ii)$\Rightarrow$(iii)$\Rightarrow$(i). Since (ii)$\Rightarrow$(iii) is obvious, we focus on showing (i)$\Rightarrow$(ii) and (iii)$\Rightarrow$(i).

\subsection{Proof of Theorem~\ref{HLproperty} (i)$\Rightarrow$(ii)}

We assume $G$ has the HL-property. Let $C\subseteq D$ be $\mathcal{T}$-free $\mathcal{L}$-structures, where $C$ is finite. For $1\leq i\leq n$, let $D_i=S_i^D$ and $C_i=S_i^C$. Then $\{D_i: 1\leq i\leq n\}$ is a partition of $D$ and $\{C_i:1\leq i\leq n\}$ is a partition of $C$. Without loss of generality, assume $D_i\neq\emptyset$ for every $1\leq i\leq n$. Then, by extending $C$, we may assume that $C_i\neq \emptyset$ for every $1\leq i\leq n$. Let $\{(C_i,a_i): 1\leq i\leq n\}$ be a natural factorization of $C$. Since $G$ acts transitively on each $D_i$, we have $D_i=G(a_i)$. By minimizing the structure on $D$, we may also assume that for any $m$-ary relation symbol $R\in\mathcal{L}$ and for any $d_1,\dots, d_m\in D$, we have $R^D(d_1,\dots,d_m)$ iff there are $c_1,\dots,c_m\in C$ and $g\in G$ such that $R^C(c_1,\dots,c_m)$ and $(d_1,\dots,d_m)=(g(c_1),\dots,g(c_m))$. 

Define $\rho: G\to \mathcal{P}(C)$ by letting, for any $g\in G$ and $c\in C$, $\rho(g)(c)=g(c)$, if $g(c)\in C$; $\rho(g)(c)$ is undefined otherwise. Since $G$ acts by isomorphisms on $D$, if $c\in C_i$ for some $1\leq i\leq n$ and $\rho(g)(c)$ is defined, then $\rho(g)(c)\in C_i$. Since $C$ is finite, the set $\rho(G)\cap \mathcal{P}_C=\{\rho(g)\in \mathcal{P}_C: g\in G\}$ is finite. 

Let $F\subseteq G$ be finite. Since the action of $G$ on $D$ is faithful, by extending $C$ with finitely many points, we may assume that $\rho(F\setminus\{1_C\})\subseteq \mathcal{P}_C$. Pick a finite $K\subseteq G$ such that $F\subseteq K$, $K^{-1}=K$ and $\rho(K\setminus\{1_C\})=\rho(G\setminus\{1_C\})\cap \mathcal{P}_C$. Define
\[
H_i=\{p_1\cdots p_l \ : \ p_1,\dots, p_l\in K \mbox{ and }\rho(p_1)(\cdots \rho(p_l)(a_i)\cdots)=a_i\}.
\]
Since $C$ and $K$ are finite, $H_i$ is finitely generated. To see this, consider an edge-labeled directed graph on $C$ defined as follows: there is an edge from $c_1$ to $c_2$ labeled by $p$ if $p\in K$ is such that $p(c_1)=c_2$. Note that this graph can have multiple edges and loops. The generators of $H_i$ are precisely those $p_1\cdots p_l$ that give a minimal cycle from $a_i$ back to $a_i$.

Let $\Gamma$ be the $\mathcal{L}$-structure with domain $G/H_1\sqcup\cdots\sqcup G/H_n$ such that for $i=1,\dots, n$, $S_i^\Gamma=G/H_i$, and for any $m$-ary relation symbol $R\in\mathcal{L}$ and for any $g_1,\dots, g_m\in G$, $R^\Gamma(g_1H_{i_1},\dots,g_mH_{i_m})$ iff there are $p_1,\dots, p_m\in K$ and $g\in G$ such that $p_j(a_{i_j})\in C_{i_j}$ for each $j=1,\dots, m$, $R^C(p_1(a_{i_1}),\cdots, p_m(a_{i_m}))$, and 
\[
(g_1H_{i_1},\dots,g_mH_{i_m})=(gp_1H_{i_1},\dots,gp_mH_{i_m}).
\]
$G$ acts on $\Gamma$ by left multiplication. Consider the map $\pi: C\to \Gamma$ defined as
$$ \pi(c)=\left\{\begin{array}{ll} H_i, & \mbox{ if $c=a_i$,} \\ pH_i, & \mbox{ if $c\in C_i, c\neq a_i$, and $p\in K$ with $p(a_i)=c$.} \end{array}\right.$$
Since $G$ acts transitively on each $D_i$, $\pi$ is well-defined. We claim that $\pi$ is an isomorphic embedding of $C$ into $\Gamma$. In fact, $\pi$ is injective because of the following fact:
\begin{enumerate}
\item[(C1)] For every $p,q\in K$ and $1\leq i\leq n$, if $p(a_i), q(a_i)\in C_i$ and $p(a_i)\neq q(a_i)$, then $p^{-1}q\notin H_i$.
\end{enumerate}
Furthermore, $\pi$ is an isomorphism between $C$ and $\pi(C)$ because of the following fact:
\begin{enumerate}
\item[(C2)] For any $p_1,\dots,p_m,q_1,\dots,q_m\in K$ such that for all $j=1,\dots, m$, $$p_j(a_{i_j}), q_j(a_{i_j})\in C_{i_j}$$ for some $1\leq i_1,\dots, i_m\leq n$, if $$R^C(p_1(a_{i_1}),\dots p_m(a_{i_m})) \mbox{ and } \neg R^C(q_1(a_{i_1}),\dots q_m(a_{i_m})),$$ then there does not exist $g\in G$ such that 
$$(p_1H_{i_1},\dots,p_mH_{i_m})=(gq_1H_{i_1},\dots,gq_mH_{i_m}).$$
\end{enumerate}
(C2) is true since otherwise in $D$ we would have 
$$R^D(p_1(a_{i_1}),\dots, p_m(a_{i_m})) \mbox{ and } \neg R^D(q_1(a_{i_1}),\dots, q_m(a_{i_m}))$$ and yet $(p_1(a_{i_1}),\dots, p_m(a_{i_m}))=(gp_1(a_{i_1}),\dots, gp_m(a_{i_m}))$, violating that $g$ is an isomorphism of the structure $D$.

Consider the map $\psi:\Gamma\to D$ defined by $\psi(gH_i)=g(a_i)$. Then $\psi$ is a homomorphism from the structure $\Gamma$ onto the structure $D$. 
Since $D$ is $\mathcal{T}$-free, so is $\Gamma$. Thus, we also have the following
\begin{enumerate}
\item[(C3)] For every structure $T\in \mathcal{T}$ there is no homomorphism from $T$ into $\Gamma$.
\end{enumerate}

Next we demonstrate that conditions (C1)--(C3) can all be equivalently expressed as certain left systems of equations on $H_1,\dots, H_n$ not having solutions. To do this, we first establish some general lemmas.

\begin{lem}\label{lem:leftsystemC1} Let $\mathcal{G}$ be a group and $\mathcal{H}\leq \mathcal{G}$. For any $\gamma, \eta\in \mathcal{G}$, $\gamma^{-1}\eta\notin \mathcal{H}$ iff the following left system with variable $x$ does not have a solution
\begin{equation}\label{leftsystemC1} \begin{array}{l} x\mathcal{H}=\gamma\mathcal{H} \\ x\mathcal{H}=\eta\mathcal{H} \end{array} \end{equation}
\end{lem}

\begin{proof} It is equivalent to state the lemma as $\gamma^{-1}\eta\in \mathcal{H}$ iff the left system (\ref{leftsystemC1}) has a solution. Now it is obvious that (\ref{leftsystemC1}) has a solution iff $\gamma\mathcal{H}=\eta\mathcal{H}$, which is equivalent to $\gamma^{-1}\eta\in \mathcal{H}$.
\end{proof}

\begin{lem}\label{lem:leftsystemC2} Let $\mathcal{G}$ be a group and $\mathcal{H}_1, \mathcal{H}_m \leq\mathcal{G}$. For any $\gamma_1, \dots, \gamma_m, \eta_1, \dots, \eta_m\in \mathcal{G}$, the following are equivalent:
\begin{enumerate}
\item[(i)] There does not exist $g\in\mathcal{G}$ such that $(\gamma_1\mathcal{H}_1, \dots, \gamma_m\mathcal{H}_m)=(g\eta_1\mathcal{H}_1, \dots, g\eta_m\mathcal{H}_m)$;
\item[(ii)] The following left system with variables $x, x_1,\dots, x_m$ does not have a solution
\begin{equation}\label{leftsystemC2} \begin{array}{l} x_1\mathcal{H}_1=\gamma_1\mathcal{H}_1 \\ x_1\mathcal{H}_1=x\eta_1\mathcal{H}_1 \\ \cdots\cdots \\ x_m\mathcal{H}_m=\gamma_m\mathcal{H}_m \\ x_m\mathcal{H}_m=x\eta_m\mathcal{H}_m\end{array} \end{equation}
\end{enumerate}
\end{lem}

\begin{proof} Again we prove the contrapositives. First, assume that there is $g\in\mathcal{G}$ such that $(\gamma_1\mathcal{H}_1, \dots, \gamma_m\mathcal{H}_m)=(g\eta_1\mathcal{H}_1, \dots, g\eta_m\mathcal{H}_m)$. Thus we have $m$ equations
\begin{equation*} \begin{array}{l} \gamma_1\mathcal{H}_1=g\eta_1\mathcal{H}_1 \\
\dots\dots \\
\gamma_m\mathcal{H}_m=g\eta_m\mathcal{H}_m
\end{array}
\end{equation*}
Each equation, which is of the form $\gamma_i\mathcal{H}_i=g\eta_i\mathcal{H}_i$, is equivalent to there existing $x_i$ such that
\begin{equation*}\begin{array}{l} x_i\mathcal{H}_i=\gamma_i\mathcal{H}_i \\ x_i\mathcal{H}=g\eta_i\mathcal{H}_i,\end{array} \end{equation*}
similar to the proof of Lemma~\ref{lem:leftsystemC1}. Thus the totality of the $m$ equations is equivalent to there existing solutions for the $2m$ equations in (\ref{leftsystemC2}). Conversely, if (\ref{leftsystemC2}) has a solution, then each pair of equations involving $x_i$ give rise to an equation of the form $\gamma_i\mathcal{H}_i=x\eta_i\mathcal{H}_i$. The solution for $x$ witnesses the existence of the desired element $g\in\mathcal{G}$ in clause (i).
\end{proof}

We are now ready to argue that conditions (C1)--(C3) can be equivalently expressed as certain left systems of equations on $H_1,\dots, H_n\leq G$ not having solutions. For (C1), simply apply Lemma~\ref{lem:leftsystemC1} to the appropriate $H_i, p, q$. Then $p^{-1}q\notin H_i$ is equivalent to the following system not having a solution
\begin{equation}\label{C1eq} \begin{array}{l} xH_i=pH_i \\ xH_i=qH_i \end{array} \end{equation}
Since $K$ is finite, there are only finitely many such systems. To summarize, there are finitely many left systems on $H_1,\dots, H_n$ such that (C1) holds iff each of the left systems does not have a solution.

For (C2), apply Lemma~\ref{lem:leftsystemC2}. The left system correspondent to the condition is
\begin{equation}\label{C2eq} \begin{array}{l} x_1H_{i_1}=p_1H_{i_1} \\ x_1H_{i_1}=xq_1H_{i_1} \\ \cdots\cdots \\ x_mH_{i_m}=p_mH_{i_m} \\ x_mH_{i_m}=xq_mH_{i_m}\end{array} \end{equation}
Again, since $K$ is finite, there are only finitely many such systems, and (C2) holds iff each of these left systems does not have a solution.

For (C3), we consider any $T\in \mathcal{T}$. Enumerate the elements of $T$ as $t_1, \dots, t_l$. Introduce variables $y_1,\dots, y_l$ correspondent to $t_1,\dots, t_l$. Suppose first there is a homomorphism of $T$ into $\Gamma$. Then there are $g_1, \dots, g_l\in G$ and $1\leq i_1, \dots, i_l\leq n$ such that, for any $m$-ary relation symbol $R\in\mathcal{L}$, whenever $R^T(t_{j_1}, \dots, t_{j_m})$ where $1\leq j_1, \dots, j_m\leq l$, we have $$R^\Gamma(g_{j_1}H_{i_{j_1}}, \dots, g_{j_m}H_{i_{j_m}}).$$ Note that $R^\Gamma(g_{j_1}H_{i_{j_1}}, \dots, g_{j_m}H_{i_{j_m}})$ iff there are $p_1, \dots, p_m\in K$ and $g\in G$ such that $p_k(a_{i_{j_k}})\in C_{i_{j_k}}$ for all $k=1,\dots, m$, 
$$ R^C(p_1(a_{i_{j_1}}), \dots, p_m(a_{i_{j_m}})), $$
and
$$ (g_{j_1}H_{i_{j_1}}, \dots, g_{j_m}H_{i_{j_m}})=(gp_1H_{i_{j_1}}, \dots gp_mH_{i_{j_m}}). $$
Applying Lemma~\ref{lem:leftsystemC2}, the above statement is equivalent to the following: there are $1\leq i_1,\dots, i_l\leq n$ such that for any $m$-ary relation symbol $R\in\mathcal{L}$, whenever $R^T(t_{j_1},\dots, t_{j_m})$ with $1\leq j_1,\dots, j_m\leq l$, there are $p_1,\dots, p_m\in K$ such that $p_k(a_{i_{j_k}})\in C_{i_{j_k}}$ for all $k=1,\dots, m$, 
$$ R^C(p_1(a_{i_{j_1}}), \dots, p_m(a_{i_{j_m}})), $$
and the following left system with variables $ y_1,\dots, y_l, x, z_1,\dots, z_m$ has a solution
\begin{equation}\label{leftsystemC3}\begin{array}{l}
z_1H_{i_{j_1}}=y_{j_1}H_{i_{j_1}}\\
z_1H_{i_{j_1}}=xp_1H_{j_1} \\
\dots\dots \\
z_mH_{i_{j_m}}=y_{j_m}H_{i_{j_m}}\\
z_mH_{i_{j_m}}=xp_mH_{i_{j_m}}
\end{array}
\end{equation}
Here the variables $x, z_1,\dots, z_m$ and the left system (\ref{leftsystemC3}) are introduced for each instance of $j_1, \dots, j_m$ and $p_1,\dots, p_m\in K$ that satisfy the conditions $R^T(t_{j_1},\dots, t_{j_m})$, $p_k(a_{i_{j_k}})\in C_{i_{j_k}}$ for all $k=1,\dots, m$, and $ R^C(p_1(a_{i_{j_1}}), \dots, p_m(a_{i_{j_m}}))$. We call these $j_1,\dots, j_m$ and $p_1,\dots, p_m\in K$ {\em a set of witnesses}. There are only finitely many possible sets of witnesses. Accumulating all sets of witnesses together, and introducing a left system (\ref{leftsystemC3}) with distinct variables $x, z_1,\dots, z_m$ for each set of witnesses, we obtain a single finite left system that is the union of all these left systems for each set of witnesses. Now this resulting left system has a solution. Conversely, if this system has a solution, then the solutions for $y_1, \dots, y_l$ will witness a homomorphism of $T$ into $\Gamma$. Thus the existence of a homomorphism of $T$ into $\Gamma$ is equivalent to a single left system having a solution.


Finally, since $\mathcal{T}$ is finite, we again have finitely many left systems such that (C3) holds iff each of the finitely many left systems on $H_1, \dots, H_n$ does not have a solution. 
 
In summary, all conditions (C1)--(C3) can be represented as finitely many left systems on $H_1,\dots, H_n$ not having a solution. Since $G$ has the HL-property, we can find $N_1,\dots,N_n\unlhd G$ such that each of the left systems described by (C1)--(C3) does not have a solution with respect to $(N_1H_1,\dots,N_nH_n)$. Indeed, for each of the left system $\Sigma$ there are such $N_1^\Sigma, \dots, N_n^\Sigma$ for the system. For each $i=1, \dots, n$, let $N_i$ be the intersection of all $N_i^\Sigma$. We thus get $N_1, \dots, N_n$ which are still of finite index in $G$ so that all of the left systems on $N_1H_1, \dots, N_nH_n$ still do not have a solution. This implies that the conditions (C1)--(C3) continue to hold with $H_i$ replaced by $N_iH_i$.

We now define $D'$ to be the finite $\mathcal{L}$-structure with domain $G/N_1H_1\sqcup\cdots\sqcup G/N_nH_n$ such that $S_i^{D'}=G/N_iH_i$ for all $i=1,\dots, n$, and for any $m$-ary relation symbol $R\in \mathcal{L}$, we have $R^{D'}(g_1N_{i_1}H_{i_1},\dots,g_mN_{i_m}H_{i_m})$ iff there are $p_1,\dots,p_m\in K$ and $g\in G$ such that $p_j(a_{i_j})\in C_{i_j}$ for all $j=1,\dots, m$, $R^C(p_1(a_{i_1}),\cdots,p_m(a_{i_m}))$, and 
\[
(g_1N_{i_1}H_{i_1},\dots,g_mN_{i_m}H_{i_m})=(gp_1N_{i_1}H_{i_1},\dots,gp_mN_{i_m}H_{i_m}).
\]
Consider the map $\pi': C\to D'$ defined as
$$ \pi'(c)=\left\{\begin{array}{ll} N_iH_i, & \mbox{ if $c=a_i$,} \\ pN_iH_i, & \mbox{ if $c\in C_i, c\neq a_i$, and $p\in K$ with $p(a_i)=c$.} \end{array}\right.$$
Then conditions (C1) and (C2) with $H_i$ replaced by $N_iH_i$ guarantee that $\pi'$ is an isomorphic embedding.  Condition (C3) with $H_i$ replaced by $N_iH_i$ implies that $D'$ is $\mathcal{T}$-free. The action of $G$ on $D'$ is by left multiplication, and each of $g\in G$ gives an isomorphism of the structure $D'$. Finally, we check that $\pi'$ is a $K$-map, and therefore an $F$-map. Let $p\in K$ and $c\in C_i$, and assume $p(c)\in C_i$. Suppose $q\in K$ with $q(a_i)=c$ and $r\in K$ with $r(a_i)=p(c)$. Then $\pi'(p(c))=rN_iH_i=pqN_iH_i=p(\pi'(c))$, where $r^{-1}pq\in H_i$ by the definition of $H_i$. This completes the proof of (i)$\Rightarrow$(ii).

\subsection{Proof of Theorem~\ref{HLproperty} (iii)$\Rightarrow$(i)}
We assume (iii) holds and show that $G$ has the HL-property. Suppose $H_1, \dots, H_n\leq G$ are finitely generated subgroups. Consider a left system $\Sigma$ with $l$ many equations on $H_1,\dots,H_n$ that does not have a solution. Let $A$ be the finite set of $g, g^{-1}\in G$ for all constants $g$ appearing in $\Sigma$. Let $H_0=\{1_G\}\leq G$ be the trivial subgroup. Consider a relational structure $D$ defined as follows: 
\begin{enumerate}
\item[a)] the domain of $D$ is $G/H_0\sqcup G/H_1\sqcup\cdots\sqcup G/H_n$;
\item[b)] there are $n+1$ many unary relation symbols $S_0, \dots, S_n$ such that $S_i^D=G/H_i$ for $i=0,\dots, n$;
\item[c)] there is a binary relation symbol $U$ such that $U^D=D\times D$;
\item[d)]for each $g\in A$, there is a binary relation $B_g$ such that $B_g^D=\{(hH_0, hgH_0): h\in G\}$;
\item[e)] for each tuple $t=(i_1,\dots, i_m)$, where $2\leq m\leq 2l+n+1$ and $0\leq i_j\leq n$ for each $j=1, \dots, m$, there is an $m$-ary relation symbol $R_t$ such that 
$$R_t^D(g_1H_{i_1},\dots,g_mH_{i_m}) \mbox{ iff } g_1H_{i_1}\cap\dots\cap g_mH_{i_m}\neq \emptyset. $$
\end{enumerate}
Let $\mathcal{L}$ be the language of $D$. We claim that the left system $\Sigma$ has a solution iff a specific finite $\mathcal{L}$-structure $T$ has a homomorphic image inside $D$. 

First we turn $\Sigma$ into an equivalent left system $\Sigma^*$ with the same number of equations. To do this, collect all equations in $\Sigma$ of the form $xH_i=gH_i$ where $x$ is a variable and $g\in A$. Introduce a new variable $y$ and replace every equation in the above collection by the equation $xH_i=ygH_i$. Denote the resulting left system as $\Sigma^*$. We claim that $\Sigma$ has a solution iff $\Sigma^*$ has a solution. First suppose $\Sigma$ has a solution. Then the solution for $\Sigma$ together with $y=1_G$ is a solution for $\Sigma^*$. Conversely, suppose $\Sigma^*$ has a solution in which $y=h$ in particular. Then this solution with every term left-multiplied by $h^{-1}$ is still a solution for $\Sigma^*$, which, with $y$ dropped, is a solution for $\Sigma$. Thus, without loss of generality, we may assume that all equations in $\Sigma$ are of the form $xH_i=ygH_i$ where $x, y$ are variables and $g\in A$.

Next we note that every equation of the form $xH_i=ygH_i$ can be replaced with two equations of the form $xH_i=x_{\text{new}}H_i$ and $x_{\text{new}}H_0=ygH_0$, where the last equation can be rearranged as $yH_0=x_{\text{new}}g^{-1}H_0$. By repeating this process, we may obtain an equivalent left system $\Sigma'$ with $\leq 2l$ many equations such that for any variable $x$ in $\Sigma$, the equations in $\Sigma'$ involving $x$ are all of the form $xH_i=yH_i$ or $xH_0=ygH_0$ for some variable $y$ and constant $g\in A$. Note that for the new variable $x_{\rm new}$ above, we get two equations $x_{\rm new}H_i=xH_i$ and $x_{\rm new}H_0=ygH_0$ by moving the cosets for $x_{\rm new}$ to the left hand side of the equations. Now for each variable $x$ in $\Sigma'$, consider the left system $\Sigma_x$ consisting only of the equations in $\Sigma'$ that involve $x$. From the above discussion we know that $\Sigma_x$ can be listed as:
$$\begin{array}{l} xH_{i_1}=\epsilon_1 H_{i_1} \\ \cdots\cdots \\ xH_{i_k}= \epsilon_k H_{i_k} \end{array} $$
for $k\leq 2l$ and each $\epsilon_j$ is either a variable $y$ or of the form $yg$ (in which case $i_j=0$) for a variable $y$ and a constant $g\in A$. Note that $\Sigma_x$ has a solution iff the following expression has a solution:
\begin{equation}\label{int} xH_0\cap xH_1\cap \cdots \cap xH_n\cap \epsilon_1H_{i_1}\cap \cdots \cap \epsilon_kH_{i_k}\neq\emptyset. \end{equation}
In fact, if $\Sigma_x$ has a solution $x, y, \dots$, then $x$ is in the intersection of (\ref{int}). Conversely, if (\ref{int}) holds for some $x, y,\dots$ then they become a solution of $\Sigma_x$. Thus each $\Sigma_x$ corresponds to a formal relation
\begin{equation}\label{formal} R_t(xH_0, xH_1, \dots, xH_n, \epsilon_1H_{i_1}, \dots, \epsilon_k H_{i_k}) \end{equation}
for a suitable $t$ of length $k+n+1\leq 2l+n+1$.

We now describe a finite $\mathcal{L}$-structure $T$. The domain of $T$ is the set of all formal cosets $xH_i$ and $xgH_0$, where $x$ is a variable in $\Sigma'$, $g\in A$, and $i=0,\dots, n$. The definition of $S_i^T$ is obvious. Also $U^T=T\times T$. For each $g\in A$, let $B_g^T=\{(xH_0, xgH_0): x \mbox{ is a variable in $\Sigma'$}\}$. The above formal relation (\ref{formal}) becomes now the definition of $R_t^T$. For other relation symbols $R_t$, $R_t^T$ is empty. Note that $T$ is a Gaifman clique. 

It is now clear that $\Sigma'$ has a solution iff there is a homomorphism from the structure $T$ into the structure $D$. Since $\Sigma$ does not have a solution, neither does $\Sigma'$ and it follows that $D$ is $T$-free.

$G$ acts faithfully on $D$ by left multiplication, and it is clear that the left multiplication by any $g\in G$ preserves the structure of $D$. It is also clear that $G$ acts transitively on $S_i^D=G/H_i$ for each $i=0, \dots, n$.

Let $C$ be a finite substructure of $D$ whose domain consists of all $H_i$ and $gH_i$ for $g\in A$ and $i=0,\dots, n$. Define $\rho: G\to \mathcal{P}(C)$ by letting $\rho(g)(c)=g(c)$ if $c, g(c)\in C$; otherwise $\rho(g)(c)$ is undefined. Since $C$ is finite, the set $\rho(G)$ is finite. Let $F\subseteq G$ be a finite subset so that $A\subseteq F$, $\rho(F)=\rho(G)$ and for each $i=1,\dots, n$, $F$ contains a finite set of generators for $H_i$. Since the action of $G$ on $S_i^D$ is transitive for each $i=0, \dots, n$, the partial action of $\rho(F)$ on $S_i^C$ is also transitive. Apply (iii) to get a finite $T$-free extension $D'$ of $C$ on which $G$ acts by isomorphisms, and an $F$-embedding $\pi$ from $C$ into $D'$. Note that $H_i$ is an element of $C$ and $\pi(H_i)$ is an element of $D'$, and we may assume that $D'=G(\pi(H_0))\sqcup\cdots\sqcup G(\pi(H_n))$. Let 
\[
N_i=\{g\in G : g(a)=a \text{ for every } a\in G(\pi(H_i))\}. 
\]
Since $D'$ is finite, $N_i$ is a normal subgroup of finite index.  Now let $\Sigma_{\vec{N}}$ be obtained from $\Sigma$ by replacing $H_1,\dots, H_n$ respectively by $N_1H_1,\dots, N_nH_n$. We claim that $\Sigma_{\vec{N}}$ does not have a solution, which shows that $G$ has the HL-property. 

Towards a contradiction, assume the left system $\Sigma_{\vec{N}}$ on $N_1H_1,\dots,N_nH_n$ has a solution. Similarly to the above, we can obtain an equivalent left system $\Sigma_{\vec{N}}^*$ such that each equation in $\Sigma_{\vec{N}}^*$ is of the form $xN_iH_i=ygN_iH_i$. Let $M_0=N_0\cap N_1\cap\cdots \cap N_n$. Then $M_0$ is still a normal subgroup of finite index, and obviously $M_0\leq N_i$ for all $i=0, \dots, n$. Now each equation of the form $xN_iH_i=ygN_iH_i$ in $\Sigma_{\vec{N}}^*$ can be equivalently replaced by $xN_iH_i=x_{\rm new}N_iH_i$ and $x_{\rm new}M_0H_0=ygM_0H_0$. Also, the last equation can be reformulated as $yM_0H_0=x_{\rm new}g^{-1}M_0H_0$ because of the normality of $M_0$. Thus we obtain an equivalent left system $\Sigma_{\vec{N}}'$ in a similar way as before, whose solution describes a homomorphic image of $T$ in a structure $\Gamma=G/M_0H_0\sqcup G/N_1H_1\sqcup\cdots\sqcup G/N_nH_n$. The exact definition of the structure $\Gamma$ is similar to the definition of $D$ above. For notational convenience we define $M_i=N_i$ for $i=1, \dots, n$.

Since $D'$ is a $T$-free structure, it is enough to show that there is a homomorphism from $\Gamma$ into $D'$. Consider the map $\psi:\Gamma\rightarrow D'$ defined by $\psi(gM_iH_i)= g(\pi(H_i))$. Then $\psi$ is the desired homomorphism. Note that $\psi$ is well-defined since if $g_1M_iH_i=g_2M_iH_i$ then $g_2=g_1nh$ for some $n\in M_i\leq N_i$ and $h\in H_i$; using the definition of $N_i$ and $H_i$ and the fact that $\pi$ is an $F$-embedding, we have $g_2(\pi(H_i))=g_1nh(\pi(H_i))=g_1(\pi(H_i))$. More precisely, we can write $h=f_1\cdots f_r$ with $f_1,\dots, f_r\in F$ as $F$ contains a finite set of generators for $H_i$; since $\pi$ is an $F$-embedding, we have $h(\pi(H_i))=f_1\cdots f_r(\pi(H_i))=\pi(f_1\cdots f_r(H_i))=\pi(H_i)$. Also, by definition of $N_i$, for $n\in N_i$ we have $n(\pi(H_i))=\pi(H_i)$. 

It remains to verify that $\psi$ preserves structure. For this let $t=(i_1,\dots, i_m)$ with $m\leq 2l+n+1$ and assume $R_t^\Gamma(g_1M_{i_1}H_{i_1},\dots,g_mM_{i_m}H_{i_m})$, that is, 
$$ g_1M_{i_1}H_{i_1}\cap \cdots\cap g_mM_{i_m}H_{i_m}\neq\emptyset. $$
Then there are $n_{i_j}\in M_{i_j}\leq N_{i_j}$ and $h_{i_j}\in H_{i_j}$ for $j=1,\dots, m$ such that 
$$g_1n_{i_1}h_{i_1}=\cdots=g_mn_{i_m}h_{i_m}=g. $$
The action of $g$ on $D'$ sends the tuple $(\pi(H_{i_1}), \dots,\pi(H_{i_m}))$ to 
\[
(g_1(\pi(H_{i_1})),\dots,g_m(\pi(H_{i_m})))=(\psi(g_1M_{i_1}H_{i_1}),\dots,\psi(g_mM_{i_m}H_{i_m})).
\]
 Note that $R_t^C(H_{i_1},\dots,H_{i_m})$ and therefore $R_t^{D'}(\pi(H_{i_1}),\dots,\pi(H_{i_m}))$. Now since $g$ acts by an isomorphism on $D'$, we have $R_t^{D'}(g_1(\pi(H_1)),\dots,g_m(\pi(H_m)))$.

Finally, consider $(hM_0H_0, hgM_0H_0)\in B_g^\Gamma$ for some $g\in A$ and $h\in G$. We need to show that $(h(\pi(H_0)), hg(\pi(H_0)))\in B_g^{D'}$. By the definition of $C$, we have $H_0, gH_0\in C$ and $B_g^C(H_0, gH_0)$. Since $g\in A\subseteq F$ and $\pi$ is an $F$-embedding, we have $\pi(gH_0)=g(\pi(H_0))$ and $B_g^{D'}(\pi(H_0), g(\pi(H_0)))$. Now $h$ acts by an isomorphism on $D'$, and so $B_g^{D'}(h\pi(H_0), hg(\pi(H_0)))$ as  desired.

This finishes the proof of Theorem~\ref{HLproperty}.

\subsection{Free products of groups with the HL-property}

As a corollary to Theorem~\ref{HLproperty}, we show below that the HL-property is closed under taking finite free products. This is analogous to the theorem of Coulbois \cite{C} which states that the RZ-property is closed under taking finite free products. In the proof of the corollary we use the coherence result of Sinora--Solecki \cite{S2}, which is also established in \cite{HKN} with a different proof. We summarize in the following proposition the exact fact we will need in our proof.

\begin{prop}\label{HKNprop}
Let $C$ be a finite $\mathcal{T}$-free $\mathcal{L}$-structure. Then, $C$ has a finite $\mathcal{T}$-free HL-extension $(D,\phi)$ such that for every substructure $E\subseteq C$, $\phi_E :\mbox{\rm Aut}(E)\to \mbox{\rm Aut}(D)$ defined as 
$$\phi_E(p)=\left\{\begin{array}{ll}\phi(p), &\mbox{ if $p\in \mbox{\rm Aut}(E)\cap \mathcal{P}_C$,} \\
1_D, & \mbox{ if $p=1_E$,} \end{array}\right. $$ is a group isomorphic embedding.
\end{prop}
\begin{proof}
It was proved in \cite{S2} and \cite{HKN} that, for any finite $\mathcal{T}$-free $\mathcal{L}$-structure $C$, there is a finite $\mathcal{T}$-free extension $D$ of $C$ and a map $\varphi: \mathcal{P}(C)\to \mbox{Aut}(D)$ such that $p\subseteq \varphi(p)$ for all $p\in \mathcal{P}(C)$, and for any $p,q,r\in \mathcal{P}(C)$ with $p\circ q=r$ we have $\varphi(p)\circ \varphi(q)=\varphi(r)$. We claim $(D,\varphi\upharpoonright \mathcal{P}_C)$ is the desired HL-extension. Let $E\subseteq C$ be a substructure. Since $1_E\circ 1_E=1_E$, we have $\varphi(1_E)\circ \varphi(1_E)=\varphi(1_E)$. Thus $\varphi(1_E)=1_D$, and $\phi_E=\varphi\upharpoonright \mbox{Aut}(E)$. The coherence property clearly implies that $\phi_E$ is a group homomorphism from $\mbox{Aut}(E)$ into $\mbox{Aut}(D)$. Assume $g\in\mbox{Aut}(E)\subseteq \mathcal{P}(C)$ and $\phi_E(g)=\varphi(g)=1_D$, then $g=1_E$ since $g\subseteq\varphi(g)$. Therefore, $\phi_E$ is an isomorphic embedding from $\mbox{Aut}(E)$ into $\mbox{Aut}(D)$.
\end{proof}

In the proof of the corollary we will also need a property of Gaifman cliques proved by Siniora--Solecki in \cite{S2}. To explain the property, first recall some definitions.

\begin{defn}
Let $\mathcal{L}$ be a relational language and $C_1,C_2$ and $C$ be $\mathcal{L}$-structures. Assume $C\subseteq C_1,C_2$. Then the \textit{free amalgamation} of $C_1$ and $C_2$ over $C$ is the structure on $D=(C_1\setminus C) \sqcup C \sqcup (C_2\setminus C)$ where for every relation $R$ in the language $R^{D}=R^{C_1}\cup R^{C_2}$. A class $\mathcal{C}$ of $\mathcal{L}$-structures has the {\em free amalgamation property} if the free amalgamation of any two structures in $\mathcal{C}$ over a structure in $\mathcal{C}$ is still in $\mathcal{C}$.
\end{defn}
Siniora--Solecki proved in Lemma 4.5 of \cite{S2} that a class $\mathcal{C}$ of $\mathcal{L}$-structures has the free amalgamation property iff there is a set $\mathcal{T}$ of $\mathcal{L}$-structures each of which is a Gaifman clique such that $\mathcal{C}$ is exactly the collection of all $\mathcal{L}$-structures $C$ for which there does not exist any isomorphic embedding from any $T\in\mathcal{T}$ into $C$.  
Note that in our context (where $\mathcal{T}$ is a finite set of finite $\mathcal{L}$-structures) the statement implies that the class of finite $\mathcal{T}$-free $\mathcal{L}$-structures has the free amalgamation property iff all $T\in\mathcal{T}$ are Gaifman cliques. This is because, if $\mathcal{T}$ is a set of Gaifman cliques and if we let $\mathcal{T}'$ to be the set of all homomorphic images of structures in $\mathcal{T}$, then $\mathcal{T}'$ is still a finite set of Gaifman cliques, and the collection of $\mathcal{T}$-free structures is exactly the collection of structures into which no $T\in\mathcal{T}'$ isomorphically embed.

\begin{cor}\label{coherenceHLproperty}
Let $G_1,G_2$ be two groups with the HL-property. Then, the free product of $G_1$ and $G_2$, $G_1*G_2$, has the HL-property.
\end{cor}

\begin{proof}
Suppose $G_1, G_2$ have the HL-property. To show that $G_1*G_2$ has the HL-property, we use the equivalence between clauses (i) and (iii) of Theorem \ref{HLproperty}. Specifically, we show the following: 
\begin{enumerate}
\item[] Let $\mathcal{L}$ be a finite relational language with unary relation symbols $S_1,\dots, S_n$. Let $\mathcal{T}$ be a finite set of finite $\mathcal{L}$-structures such that every $T\in\mathcal{T}$ is a Gaifman clique. Let $D$ be a $\mathcal{T}$-free $\mathcal{L}$-structure such that $\{S_1^D, \dots, S_n^D\}$ is a partition of the domain of $D$. Let $C$ be a finite substructure of $D$. Let $F$ be a finite subset of $G_1*G_2$. Suppose that $G_1*G_2$ acts faithfully by isomorphisms on $D$ and that $G_1*G_2$ acts transitively on each $S_i^D$ for $i=1,\dots, n$. Then there exists a finite $\mathcal{T}$-free $\mathcal{L}$-structure $D'$ on which $G_1*G_2$ acts by isomorphisms, and an $F$-embedding from $C$ into $D'$.
\end{enumerate}
In the following we construct the desired structure $D'$.

Let $F_1\subseteq G_1$ and $F_2\subseteq G_2$ be finite subsets such that $F\subseteq F_1*F_2$. Let $C'\subseteq D$ be a finite structure extending $C$ such that for every $f=f_1f_2\cdots f_l\in F$ where $f_i\in F_1\cup F_2$ for every $i=1,2,\dots,l$, and every $a\in C$ where $f (a)\in C$, we have $f_j\cdots f_l (a)\in C'$ for every $1\leq j\leq l$. Since $G_1$ and $G_2$ have the HL-property, we can find finite $\mathcal{T}$-free $\mathcal{L}$-structures $D_1'$ and $D_2'$ such that for $k=1,2$:
\begin{enumerate}
\item $G_k$ acts by isomorphisms on $D_k'$, and
\item there exists an $F_k$-embedding $\pi_k$ from $C'$ to $D_k'$.
\end{enumerate}

Let $D_0$ be the free amalgamation of $D_1'$ and $D_2'$ over $\pi_1(C')\cong \pi_2(C')$, that is, the underlying set of $D_0$ is $(D_1'\setminus \pi_1(C'))\sqcup C'\sqcup (D_2'\setminus \pi_2(C'))$ and for every relation $R$ in the language $R^{D_0}=R^{D_1'}\cup R^{D_2'}$. Since $\mathcal{T}$ consists of only Gaifman cliques, the collection of all $\mathcal{T}$-free $\mathcal{L}$-structures has the free amalgamation property. Thus $D_0$ is $\mathcal{T}$-free.

By Proposition \ref{HKNprop}, there exists a finite $\mathcal{T}$-free HL-extension $(D',\phi)$ of $D_0$ such that for every finite substructure $E\subseteq D_0$, $\phi$ induces a group isomorphic embedding from $\mbox{Aut}(E)$ to $\mbox{Aut}(D)$. In particular, this holds for $E=D_1',D_2'$. Therefore, $\phi$ induces an action of $G_k$ on $D'$ by $g(a)=\phi(g)(a)$ for $k=1,2$. By considering the free product of these two actions, we get an action of $G_1*G_2$ on $D'$ by isomorphisms. It remains to show that there exists an $F$-embedding $\pi$ from $C$ to $D'$. Let $\pi:C'\rightarrow D'$ denote the inclusion map. We claim $\pi\upharpoonright C$ is as desired. Let $f=f_1f_2\cdots f_l\in F$ where $f_i\in F_1\cup F_2$ for every $i=1,2,\dots,l,$ and $a\in C$ be such that $f(a)\in C$. Note that for $k=1, 2$, since $\pi$ is an $F_k$-embedding from $C'$ to $D_k'$, we have that $\pi$ is also an $F_k$-embedding from $C'$ to $D'$. Therefore, 
\[
\pi(f(a))=\pi(f_1\cdots f_k (a))=f_1(\pi(f_2\cdots f_l (a)))=\cdots=f_1\cdots f_l (\pi(a))=f(\pi(a)).
\]
\end{proof}

\section{Coherent HL-extensions and Ultraextensive Structures\label{CoherentHLextensions}}

In this section we introduce a notion of ultraextensive $\mathcal{L}$-structures using a new notion of coherent HL-extensions. Coherence in our sense is slightly weaker than the coherence notion of Siniora--Solecki \cite{S2} but is sufficient for deriving the interesting properties of ultraextensive structures. These notions are generalizations of similar notions in \cite{EG} in the context of metric spaces.

\begin{defn}\label{coherence}
Let $C_1\subseteq C_2$ be $\mathcal{L}$-structures and $(D_i,\phi_i)$ be an HL-extension of $C_i$ for $i=1,2$. We say that $(D_1,\phi_1)$ and $(D_2, \phi_2)$ are {\it coherent} if 
\begin{enumerate}
\item[(i)] $D_2$ extends $D_1$, 
\item[(ii)] $\phi_2(p)$ extends $\phi_1(p)$ for all $p\in\mathcal{P}_{C_1}\subseteq \mathcal{P}_{C_2}$, and
\item[(iii)] letting $K_i=\langle\phi_i(\mathcal{P}_{C_i})\rangle\leq \mbox{Aut}(D_i)$ for $i=1,2$, and letting $\kappa:\phi_1(\mathcal{P}_{C_1})\to \phi_2(\mathcal{P}_{C_2})$ be such that $\kappa(\phi_1(p))=\phi_2(p)$ for all $p\in\mathcal{P}_{C_1}$, then $\kappa$ has a unique extension to a group isomorphic embedding from $K_1$ into $K_2$.
\end{enumerate}
\end{defn}

\begin{defn}\label{uedef}
An $\mathcal{L}$-structure $U$ is {\it ultraextensive} if 
\begin{enumerate}
\item[(i)] $U$ is ultrahomogeneous, i.e., there is a $\phi$ such that $(U,\phi)$ is an HL-extension of $U$;
\item[(ii)] Every finite $C\subseteq U$ has a finite HL-extension $(D,\phi)$ where $D\subseteq U$;
\item[(iii)] If $C_1\subseteq C_2\subseteq U$ are finite and $(D_1, \phi_1)$ is a finite minimal HL-extension of $C_1$ with $D_1\subseteq U$, then there is a finite minimal HL-extension $(D_2,\phi_2)$ of $C_2$ such that $D_2\subseteq U$ and $(D_1,\phi_1)$ and $(D_2,\phi_2)$ are coherent.
\end{enumerate}
\end{defn}

\begin{thm}\label{thmcoherent} Let $\mathcal{T}$ be a finite set of finite $\mathcal{L}$-structures each of which is a Gaifman clique. Suppose $C_1\subseteq C_2$ are finite $\mathcal{T}$-free $\mathcal{L}$-structures and $(D_1,\phi_1)$ is a finite $\mathcal{T}$-free HL-extension of $C_1$. Then there is a finite $\mathcal{T}$-free HL-extension $(D_2,\phi_2)$ of $C_2$ so that $(D_2, \phi_2)$ is coherent with $(D_1,\phi_1)$. 
\end{thm}

\begin{proof}
Since every $T\in\mathcal{T}$ is a Gaifman clique, the collection of all $\mathcal{T}$-free structures has the free amalgamation property. Let $C$ be the free amalgamation of $D_1$ and $C_2$ over $C_1$. Then $C$ is $\mathcal{T}$-free. We will again use the main theorem of \cite{S2} and \cite{HKN}, which states that, for any finite $\mathcal{T}$-free $\mathcal{L}$-structure $C$, there is a finite $\mathcal{T}$-free extension $D_2$ of $C$ and a map $\varphi: \mathcal{P}(C)\to \mbox{Aut}(D_2)$ such that $p\subseteq \varphi(p)$ for all $p\in \mathcal{P}(C)$, and for any $p,q,r\in \mathcal{P}(C)$ with $p\circ q=r$ we have $\varphi(p)\circ \varphi(q)=\varphi(r)$. 

Define $\phi_2: \mathcal{P}_{C_2}\to \mbox{Aut}(D_2)$ as
$$ \phi_2(p)=\left\{\begin{array}{ll} \varphi(\phi_1(p)), & \mbox{ if $p\in \mathcal{P}_{C_1}\subseteq \mathcal P_{C_2}$,}\\
\varphi(p), & \mbox{ if $p\in \mathcal{P}_{C_2}\setminus \mathcal{P}_{C_1} \subseteq \mathcal{P}(C)$. }
\end{array}\right.
$$
Then $(D_2, \phi_2)$ is an HL-extension of $C_2$. It is also clear that $D_2$ extends $D_1$. For $p\in \mathcal{P}_{C_1}$, our definition of $\phi_2$ gives that $\phi_2(p)=\varphi(\phi_1(p))\supseteq \phi_1(p)$. Now define $\kappa: K_1\to K_2$ by letting $\kappa(\phi_1(p))=\phi_2(p)$ and extending the definition of $\kappa$ to all finite products in $K_1=\langle \phi_1(\mathcal{P}_{C_1})\rangle\leq \mbox{Aut}(D_1)$. We first verify that $\kappa$ is well-defined. For this let $p_1,\dots, p_n\in \mathcal{P}_{C_1}$ such that $\phi_1(p_1)\cdots \phi_1(p_n)=1_{K_1}$. We need to show that $\phi_2(p_1)\cdots\phi_2(p_n)=1_{K_2}$. Both products take place in an automorphism group, so they are compositions. By the coherent property of $\varphi$, we have $\varphi(\phi_1(p_1))\cdots\varphi(\phi_1(p_n))=\varphi(1_{K_1})$, and so $\phi_2(p_1)\cdots \phi_2(p_n)=1_{K_2}$. Thus we have shown that $\kappa$ is a group homomorphism. To see that it is a group isomorphic embedding, we show that the kernel of $\kappa$ is trivial. For this let $p_1,\dots, p_n\in \mathcal{P}_{C_1}$ so that $\phi_2(p_1)\cdots \phi_2(p_n)=1_{K_2}$. Restricting all maps on $D_1$, we get $\phi_1(p_1)\cdots \phi_1(p_n)=1_{K_1}$.
\end{proof}

We remark that the condition in the above theorem for $\mathcal{T}$ to consist only of Gaifman cliques is necessary. If $\mathcal{T}$ fails this property, not only the proof fails to work because of the failure of the free amalgamation property for the collection of $\mathcal{T}$-free $\mathcal{L}$-structures, but also the statement of the theorem can fail. 

We give a counterexample below.

 Consider $\mathcal{L}=\{R, S\}$ where $R$ is a binary relation symbol and $S$ is a quarternary relation symbol. Let $T=\{ 0, 1,2, 3, 4, 5,6\}$ where $R^T=\{(0,1), (1,2), (2,0)\}$ and $S^T=\{(a, b, c, d): a, b, c, d\in\{3, 4, 5, 6\}\}$. Let $C_2=\{ x, y, z\}$ with $R^{C_2}=\{(x,y), (y, z), (z, x)\}$ and $S^{C_2}=\emptyset$. Let $C_1=\{x, y\}$ be the induced substructure of $C_2$. Let $D_1=\{x, y, u, v\}$ where $R^{D_1}=\{(x, y), (y, u), (u, v), (v, x)\}$ and $S^{D_1}=\{(a, b, c, d): a, b, c, d\in\{x, y, u, v\}\}$. Then $\mathcal{P}_{C_1}=\{x\mapsto y, y\mapsto x\}$ and $(D_1,\phi_1)$ is an HL-extension of $C_1$, with $\phi_1: \mathcal{P}_{C_1}\to \mbox{Aut}(D_1)$ extending $x\mapsto y$ to the automorphism $\{x\mapsto y, y\mapsto u, u\mapsto v, v\mapsto x\}$ and extending $y\mapsto x$ to the automorphism $\{y\mapsto x, x\mapsto v, v\mapsto u, u\mapsto y\}$. Note that $C_1,C_2,D_1$ are $T$-free $\mathcal{L}$-structures. Now there is no $T$-free HL-extension $(D_2,\phi_2)$ of $C_2$ that is coherent with $(D_1, \phi_1)$.

\begin{thm} \label{uecountable} Let $\mathcal{T}$ be a finite set of finite $\mathcal{L}$-structures each of which is a Gaifman clique. Then every countable $\mathcal{T}$-free $\mathcal{L}$-structure can be extended to a countable $\mathcal{T}$-free ultraextensive  $\mathcal{L}$-structure.
\end{thm}
\begin{proof}
Let $C$ be a countable $\mathcal{T}$-free $\mathcal{L}$-structure. Write $C$ as an increasing union of finite $\mathcal{T}$-free $\mathcal{L}$-structures $F_n$ for $n=1,2,\dots$. For $n\geq 1$, inductively define increasing sequences of finite $\mathcal{T}$-free $\mathcal{L}$-structure $C_n$, $D_n$ and $Z_n$ as follows. Let $C_1=F_1$ and $(D_1, \phi_1)$ be a finite $\mathcal{T}$-free, minimal HL-extension of $C_1$. We define $Z_1\supseteq D_1$ such that for every pair $(D, D')$ with $D\subseteq D'\subseteq D_1$ and any minimal HL-extension $(E,\phi)$ of $D$ where $E\subseteq D_1$, there exists a $\mathcal{T}$-free minimal HL-extension $(E',\phi ')$ of $D'$ where $E'\subseteq Z_1$, such that $(E,\phi)$ and $(E',\phi ')$ are coherent. Note that this is possible since there are only finitely many triples $(D,D',E)$ and for any such triple by Theorem \ref{thmcoherent} we can fix a coherent extension $E'$. Finally, to construct $Z_1$, we add $E'\setminus E$ to $D_1$ for all $E'$ corresponding to the triple $(D,D',E)$ such that the union of the new points ($E'\setminus E$) and $E\subseteq D_1$ is an isomorphic copy of $E'$. $Z_1$ is a free amalgamation of $\mathcal{T}$-free structures, and hence is $\mathcal{T}$-free. Let $C_2$ be the free amalgamation of $Z_1$ and $F_2$ over $F_1$.

 In general, assume a finite $C_n$ has been defined for $n>1$. Apply Theorem~\ref{thmcoherent} to obtain a finite $\mathcal{T}$-free, minimal HL-extension $(D_n, \phi_n)$ of $C_n$ that is coherent with $(D_{n-1},\phi_{n-1})$. We use a similar construction to the construction of $Z_1$ from $D_1$ to define $Z_n\supseteq D_n$. Note that $Z_n$ has the property that for every minimal HL-extension in $D_n$, that is, for every $D,E\subseteq D_n$ where $(E,\phi)$ is a minimal HL-extension of $D$, every $D\subseteq D'\subseteq D_n$ has a minimal HL-extension in $Z_n$ that is coherent with $(E,\phi)$. Let $C_{n+1}$ be the free amalgamation of $Z_n$ and $F_{n+1}$ over $F_n$. All structures obtained are $\mathcal{T}$-free.

Let $D$ be the union of the increasing sequence $(D_n)_{n=1}^\infty$. We verify that $D$ is ultraextensive. To verify Definition~\ref{uedef} (i), let $p\in\mathcal{P}_D$. Then there is $n\geq 1$ such that $p\in \mathcal{P}_{C_n}$. Let $n_p$ be the least such $n$. Then for all $m\geq n_p$, $p\subseteq \phi_m(p)\subseteq \phi_{m+1}(p)$ by the coherence of $(D_m,\phi_m)$ with $(D_{m+1}, \phi_{m+1})$. Define $\phi(p)=\bigcup_{m\geq n_p}\phi_m(p)$. Then $\phi(p)$ is an isomorphism of $D$ that extends $p$. 

For Definition~\ref{uedef} (ii), let $F\subseteq D$ be finite. Then there is $n$ such that $F\subseteq C_n$, and it follows that $(D_n, \phi_n\upharpoonright \mathcal{P}_F)$ is an HL-extension of $F$.

Finally, for Definition~\ref{uedef} (iii), let $F\subseteq F'\subset D$ be finite and assume that $(E, \phi)$ is a finite minimal HL-extension of $F$ with $E\subseteq D$. Then, there is a natural number $n$ such that $F',E\subseteq D_n$. By the construction of $Z_n$, there exists a minimal HL-extension $(E',\phi')$ of $F'$ (corresponding to the triple $(F,F',E)$) such that $E'\subseteq Z_n\subseteq D$ and that $(E',\phi')$ is coherent with $(E, \phi)$. \qedhere
\end{proof}

We derive some properties of ultraextensive structures below.

\begin{thm} \label{uectbl}
If $U$ is an ultraextensive $\mathcal{L}$-structure, then every countable substructure $C\subseteq U$ can be extended to a countable ultraextensive substructure $D\subseteq U$.
\end{thm}
\begin{proof}
We use a similar argument to the argument in the proof of Theorem \ref{uecountable} to construct $D$. The differences are that in the construction instead of applying Theorem~\ref{thmcoherent} we use the properties of ultrextensive structures to find $(D_n,\phi_n)$; and we consider union of structures instead of free amalgamation to find $C_n,Z_n$. Clearly, all the structures $C_n,D_n,Z_n$ are substructures of $U$ and therefore, $D\subseteq U$.
\end{proof}

\begin{thm} 
If $U$ is a countable ultraextensive $\mathcal{L}$-structure then $\mbox{\rm Aut}(U)$ has a dense locally finite subgroup.
\end{thm}

\begin{proof}
Let $\{C_i\}_{i=1}^\infty$ be an increasing sequence of finite substructures of $U$ such that $U=\bigcup_{i=1}^\infty C_i$. Since $U$ is an ultraextensive $\mathcal{L}$-structure, we can find an increasing sequence $\{(D_i,\phi_i)\}_{i=1}^\infty$, where each $D_i\subseteq U$, such that $(D_i,\phi_i)$ is an HL-extension of $C_i$ and $(D_{i+1}, \phi_{i+1})$ is coherent with $(D_i,\phi_i)$ for $i=1,2,\dots$. Then, $\bigcup_{i=1}^\infty \mbox{Aut}(D_i)$ is a dense locally finite subgroup of $\mbox{Aut}(U)$.
\end{proof}

\begin{defn}\label{coherentextension}
Let $\mathcal{C}$ be a class of $\mathcal{L}$-structures. We say $\mathcal{C}$ has \textit{the coherent extension property} if it has the EPPA and for finite structures $D\subseteq D'$ in $\mathcal{C}$ and a finite minimal HL-extension $(E,\phi)$ of $D$ where $E$ is also in $\mathcal{C}$, there exists a finite minimal HL-extension $(E',\phi ')$ of $D'$ where $E'$ is in $\mathcal{C}$ and $(E,\phi)$ and $(E',\phi ')$ are coherent. 
\end{defn}

\begin{thm} 
Let $\mathcal{C}$ be a Fra\"{i}ss\'{e} class and $U$ be the Fra\"{i}ss\'{e} limit of $\mathcal{C}$. Then, $U$ is ultraextensive iff $\mathcal{C}$ has the coherent extension property. In particular, if $\mathcal{T}$ is a finite set of Gaifman cliques and $\mathcal{C}$ is the class of $\mathcal{T}$-free structures, then $U$ is ultraextensive.
\end{thm}
\begin{proof}
The equivalence is clear by Definition \ref{coherentextension}. The second part is the direct consequence of Theorem \ref{thm_HL} and Theorem \ref{thmcoherent}.
\end{proof}

\begin{cor}
The Henson graph $G_n$, the Fra\"{i}ss\'{e} limit of the class of $K_n$-free graphs, is ultraextensive for every natural number $n$.
\end{cor}

\end{document}